\documentclass[12pt]{amsart}
\usepackage{wrapfig}
\usepackage{graphicx}
\usepackage[all]{xy}
\usepackage{amsmath,amsthm,amssymb}
\usepackage{mathabx,epsfig}

\usepackage[margin=1in]{geometry}
\usepackage{amsmath,amscd}
\usepackage{mathtools}

\usepackage[all]{xy}

\usepackage{thmtools}
\usepackage{thm-restate}
\usepackage{hyperref}
\usepackage{cleveref}

\newtheorem{theorem}{Theorem}
\newtheorem{conjecture}{Conjecture}
\newtheorem{corollary}{Corollary}

\newtheorem{proposition}{Proposition}
\usepackage[small,compact]{titlesec}
\usepackage{enumerate}
\usepackage{enumitem}
\usepackage{times}
\makeatletter

\newcommand{\Rmnum}[1]{\expandafter\@slowromancap\romannumeral #1@}
\makeatother

 \newcommand{\genus}{\operatorname{genus}}

\newcommand{\Mod}[1]{\ (\textup{mod}\ #1)}
 \newcommand{\Aut}{\operatorname{Aut}}

\newcommand{\Gal}{\operatorname{Gal}}

\usepackage{amssymb,fge}
\newcommand{\mysetminus}{\mathbin{\fgebackslash}}
\newcommand{\Spec}{\operatorname{Spec}}

\usepackage{xcolor}
\usepackage{titlesec}
\titleformat{\section}[block]{\color{black}\large\filcenter}{}{1em}{}
\titleformat{\subsection}[hang]{\bfseries}{}{1em}{}
\theoremstyle{remark}
\newtheorem{definition}{Definition} 
\newtheorem{remark}{Remark}
\begin{document}
\title{Prime divisors in polynomial orbits over function fields}
\author{Wade Hindes}
\date{\today}
\maketitle
\renewcommand{\thefootnote}{}
\begin{abstract} Given a polynomial $\phi$ over a global function field $K/\mathbb{F}_q(t)$ and a wandering base point $b\in K$, we give a geometric condition on $\phi$ ensuring the existence of primitive prime divisors for almost all points in the orbit $\mathcal{O}_\phi(b):=\{\phi^n(b)\}_{n\geq0}$. As an application, we prove that the Galois groups (over $K$) of the iterates of many quadratic polynomials are large and use this to compute the density of prime divisors of $\mathcal{O}_\phi(b)$.      
\end{abstract}
\begin{section}{Introduction}
From elliptic divisibility sequences to the Fibonacci numbers, it is an important problem in number theory to prove the existence of ``new" prime divisors of an arithmetically defined sequence. For example, such ideas have applications ranging from the undecidability of Hilbert's $10$th problem \cite{Poonen}, to the classification of certain families of subgroups of finite linear groups \cite{Feit,primcycl,Linear,trans,alg}. In this paper, we study the set of prime divisors of polynomial recurrence sequences defined by iteration over global function fields. 

To wit, let $K/\mathbb{F}_q(t)$ be a finite extension, let $V_K$ be a complete set of valuations on $K$, and let $\mathcal{B}=\{b_n\}\subseteq K$ be any sequence. We say that $v\in V_K$ is a \emph{primitive prime divisor} of $b_n$ if 
\[ v(b_n)>0\;\;\;\text{and}\;\;\; v(b_m)=0\;\; \text{for all}\; 1\leq m\leq n-1. \]
Likewise, we define the \emph{Zsigmondy set of} $\mathcal{B}$ to be 
\[\mathcal{Z}(\mathcal{B}):=\big\{n\geq1\;\big\vert\; b_n\;\text{has no primitive prime divisors}\big\}.\]
Over number fields, there are numerous results regarding the finiteness (and size) of $\mathcal{Z}(\mathcal{B})$; for example, see \cite{PrimDiv,Tucker,Silv-Ing, Krieger, Silv-Vojta,SilvPrimDiv}.  

In this paper, we are interested in studying the finiteness of $\mathcal{Z}(\phi,b):=\mathcal{Z}(\mathcal{O}_\phi(b))$, where $\mathcal{O}_\phi(b):=\{\phi^n(b)\}_{n\geq0}$ is the \emph{orbit} of $b\in K$ for $\phi\in K[x]$; here the superscript $n$ denotes iteration (of $\phi$). 

The key geometric notion, allowing us to use techniques in the theory of rational points on curves over $K$ to study $\mathcal{Z}(\phi,b)$, is the following:  
\begin{definition} Let $\phi\in K(x)$ and let $\ell\geq2$ be an integer. Then we say that $\phi$ is \emph{dynamically $\ell$-power non-isotrivial} if there exists an integer $m\geq1$ such that, 
\begin{equation}{\label{Curve}} C_{m,\ell}(\phi):=\big\{(X,Y)\in \mathbb{A}^2(\bar{K})\;\big\vert\;Y^\ell=\phi^m(X)=(\underbracket{\phi\circ\phi\dots \circ\phi}_m)(X)\big\} 
\end{equation} 
is a non-isotrivial curve \cite{isotrivial} of genus at least $2$.         
\end{definition}
Similarly, we have the following refined notions of primitive prime divisors and Zsigmondy sets:  
\begin{definition} Let $\phi\in K(x)$, let $b\in K$, and let $\ell$ be an integer. We say that a place $v\in V_K$ is an \emph{$\ell$-primitive prime divisor} for $\phi^n(b)$ if all of the following conditions are satisfied: 
\begin{enumerate}[itemsep=1.1mm]
\item[\textup{(1)}] $v(\phi^n(b))>0$,
\item[\textup{(2)}] $v(\phi^m(b))=0$ for all $1\leq m\leq n-1$ such that $\phi^m(b)\neq0$,
\item[\textup{(3)}] $v(\phi^n(b))\not\equiv0\Mod{\ell}$.    
\end{enumerate} 
Moreover, we call  
\begin{equation}{\label{Zig}} \mathcal{Z}(\phi,b,\ell):=\big\{n\;\big\vert\; \phi^n(b)\; \text{has no $\ell$-primitive prime divisors}\big\}
\end{equation} the \emph{$\ell$-th Zsigmondy set} for $\phi$ and $b$.  
\end{definition} 
Note that $\mathcal{Z}(\phi,b)\subseteq\mathcal{Z}(\phi,b,\ell)$ for all $\ell$. Hence, it suffices to show that $\mathcal{Z}(\phi,b,\ell)$ is finite for a single $\ell$ to ensure that all but finitely many elements of $\mathcal{O}_\phi(b)$ have primitive prime divisors. Moreover, we use the notions of height $h_K$ and canonical height $\hat{h}_\phi$ found in \cite{Baker}.   
\begin{theorem}{\label{PrimDivThm}} Suppose that $\phi\in K[x]$, $b\in K$, and $\ell\geq2$ satisfy the following conditions: 
\begin{enumerate}[itemsep=1.1mm]
\item[\textup{(1)}] $\phi$ is dynamically $\ell$-power non-isotrivial,
\item[\textup{(2)}] $b$ is wandering (i.e. $\hat{h}_\phi(b)>0$). 
\end{enumerate}
Then $\mathcal{Z}(\phi,b,\ell)$ and $\mathcal{Z}(\phi,b)$ are finite. In particular, all but finitely many elements of $\mathcal{O}_\phi(b)$ have primitive prime divisors.     
\end{theorem} 
In addition to determining whether or not a sequence has primitive prime divisors, it is interesting to compute the ``size" of its set of prime divisors (in terms of density) \cite{Hasse, Lagarias}, a problem which has applications to the dynamical Mordell-Lang conjecture \cite{Mordell-Lang} and to questions regarding the size of hyperbolic Mandelbrot sets \cite{RafeThesis}.   

To do this, let $\mathcal{O}_K$ be the integral closure of $\mathbb{F}_q[t]$ in $K$ and let $\mathfrak{q}\subseteq \mathcal{O}_K$ be a prime ideal, determining a valuation $v_\mathfrak{q}$ on $K$. For such $\mathfrak{q}$, define the \emph{norm} of $\mathfrak{q}$ to be $N(\mathfrak{q}):=\#(\mathcal{O}_K/\mathfrak{q}\mathcal{O}_K)$, and let $\delta(\mathcal{P})$ be the \emph{Dirchlet density} of a set of primes $\mathcal{P}$ of $K$:
\begin{equation} \delta(\mathcal{P}):=\lim_{s\rightarrow 1^+}\frac{\sum_{\mathfrak{q}\in\mathcal{P}}N(\mathfrak{q})^{-s}}{\sum_{\mathfrak{q}}N(\mathfrak{q})^{-s}}
\end{equation}
We use Theorem \ref{PrimDivThm} and ideas from the Galois theory of iterates to compute the density of 
\[\mathcal{P}_\phi(b):=\big\{\mathfrak{q}\in\Spec(\mathcal{O}_K)\;\big\vert\; v_\mathfrak{q}(\phi^n(b))>0\;\text{for some $n\geq0$}\big\},\]
the set of prime divisors of the orbit $\mathcal{O}_\phi(b)$. In particular, we establish a version of \cite[Conj. 3.11]{R.Jones}; see \cite[Theorem 1]{uniformity} for the corresponding statement in characteristic zero (with uniform bounds) and \cite{B-J,R.Jones} for introductions to dynamical Galois theory.   
\begin{corollary}{\label{Galois}} Let $K/\mathbb{F}_q(t)$ for some odd $q$ and let $\phi\in K[x]$ be a quadratic polynomial.\\ Write $\phi(x)=(x-\gamma)^2+c$ and suppose that $\phi$ satisfies the following conditions: 
\begin{enumerate}[itemsep=1.5mm]
\item[\textup{(a)}] $\phi$ is not post-critically finite (i.e. $\gamma$ is wandering), 
\item[\textup{(b)}] the adjusted critical orbit $\widebar{\mathcal{O}}_\phi(\gamma)=\{-\phi(\gamma),\phi^n(\gamma)\}_{n\geq2}$ contains no squares in K,
\item[\textup{(c)}] the $j$-invariant of the elliptic cure $E_\phi: Y^2=(X-c)\cdot\phi(X)$ is non-constant. 
\end{enumerate}
Then all of the following statements hold:   
\begin{enumerate}[itemsep=1.5mm]
\item[\textup{(1)}] $\mathcal{Z}(\phi,b,2)$ is finite for all wandering points $b\in K$,  
\item[\textup{(2)}] $G_\infty(\phi)\leq\Aut(T(\phi))$ is a finite index subgroup, 
\item[\textup{(3)}] $\delta(\mathcal{P}_\phi(b))=0$\; for all $b\in K$.  
\end{enumerate} 
\end{corollary}
\begin{remark} One expects similar statements to hold for $\phi(x)=x^\ell+c$ and $\ell$ a prime. Namely,  if one can show that $\phi$ is dynamically $\ell$-power non-isotrivial, then $\mathcal{Z}(\phi,0,\ell)$ is finite by Theorem \ref{PrimDivThm} and $G_\infty(\phi)$ is a finite index subgroup of $\Aut(T(\phi))$ by \cite[Theorem 25]{Eventually}.    
\end{remark}
In particular, we apply Theorem \ref{PrimDivThm} and Corollary \ref{Galois} to the explicit family 
\[\phi_{(f,g)}(X)=\big(X-f(t)\cdot g(t)^d\big)^2+g(t)\;\;\;\;\text{for}\;\;\; f\in\mathbb{F}_q(t),-g\notin(\mathbb{F}_q(t))^2,\;\text{and}\;d\geq1.\] Note that by letting $f=0$ and $g=t$, we recover the main result of \cite{RafeThesis}. 
\begin{corollary}{\label{eg}} Let $K=\mathbb{F}_q(t)$ and let $\phi:=\phi_{(f,g)}$ be such that the $j$-invariant of the elliptic curve \[E_{\phi_{(f,g)}}: Y^2=(X-g(t))\cdot \phi_{(f,g)}(X)\] is non-constant. Then $\mathcal{Z}(\phi,b,2)$ is finite for all wandering base points, $G_\infty(\phi)\leq\Aut(T(\phi))$ is a finite index subgroup, and $\delta(\mathcal{P}_\phi(b))=0$ for all $b\in K$.    
\end{corollary}      
\end{section} 
\begin{section}{Primitive Prime Divisors and Superelliptic Curves}
To prove Theorem \ref{PrimDivThm}, we build on our techniques from the characteristic zero setting \cite{uniformity}. There, among other things, we prove the uniform bound 
\begin{equation} \#\mathcal{Z}(\phi,\gamma,2)\leq17
\end{equation} 
for all $\phi(x)=(x-\gamma(t))^2+c(t)$ satisfying $\deg(\gamma)\neq \deg(c)$; here $\gamma$ and $c$ are polynomials with coefficients in any field of characteristic zero. Additionally, we adapt ideas from \cite{Tucker} and use linear height bounds for rational points on non-isotrivial curves \cite{Kim} to prove our results.     
\begin{proof}[Proof (Theorem)] Throughout the proof, it will be useful to consider only $\ell$-primitive prime divisors avoiding some finite subset $S\subseteq V_K$. Therefore, we make the following convention:
\begin{equation} 
\mathcal{Z}(\phi,b,S,\ell):=\big\{n\;\big\vert\; \phi^n(b)\; \text{has no $\ell$-primitive prime divisors}\; v\in V_K\mysetminus S\big\}.
\end{equation}   
Clearly $\mathcal{Z}(\phi,b,\ell)\subseteq\mathcal{Z}(\phi,b,\ell,S)$ for all $S$, and so it suffices to show that $\mathcal{Z}(\phi,b,\ell,S)$ is finite for some $S\subseteq V_K$. Note that there is no harm in enlarging $S$. In particular, we may assume that 
\[\text{(a)}.\;\;\;b\in\mathcal{O}_{K,S}\;\;\;\;\;\; \text{(b)}.\;\;\;\phi\in\mathcal{O}_{K,S}[x]\;\;\;\;\;\;\text{(c)}.\;\;\; v(a_d)=0\;\;\text{for all}\; v\notin S\;\;\;\;\;\; \text{(d)}.\;\;\; \mathcal{O}_{K,S}\;\text{is a UFD,}\] where $a_d$ is the leading term of $\phi$. Note that condition (d) is made possible by \cite[Prop. 14.2 ]{Rosen}. Similarly, we see that 
\[\mathcal{Z}(\phi,b,S,\ell)\subseteq\mathcal{Z}(\phi,\phi^n(b),S,\ell)\cup\{t\in\mathbb{Z}\;\vert\; 1\leq t\leq n\}\;\;\,\text{for all}\;n\geq0.\]
Therefore, after replacing $b$ with $\phi^n(b)$ for some $n$, we may assume that $0\notin\mathcal{O}_\phi(b)$. By the assumptions on $S$ above, we see that $\phi^n(b)\in\mathcal{O}_{K,S}$ for all $n$, permitting us to write 
\begin{equation}{\label{decomp}} \phi^n(b)=u_n\cdot d_n\cdot y_n^\ell,\;\;\text{for some}\;\;\; d_n,y_n\in\mathcal{O}_{K,S}\;,\;u_n\in\mathcal{O}_{K,S}^*. 
\end{equation} 
However, since $\mathcal{O}_{K,S}^*$ is a finitely generated group \cite[Prop. 14.2 ]{Rosen}, we may write $u_n=\textbf{u}_1^{r_1}\cdot \textbf{u}_2^{r_2}\dots \textbf{u}_t^{r_t}$ for some basis $\{\textbf{u}_i\}$ of $\mathcal{O}_{K,S}^*$ and some integers $0\leq r_i\leq\ell-1$. In particular, the height $h_K(u_n)$ is bounded independently of $n\geq0$. Similarly, we may assume that $0\leq v(d_n)\leq\ell-1$ for all $v\notin S$. To see this, we use the correspondence $V_K\mysetminus S\longleftrightarrow\Spec(\mathcal{O}_{K,S})$ discussed in \cite[Ch. 14]{Rosen} and the fact that $\mathcal{O}_{K,S}$ is a UFD to write 
\[d_n=p_1^{e_1}\cdot p_2^{e_2}\cdots p_s^{e_s}\big(p_1^{q_1}\cdot p_2^{q_2}\cdots p_s^{q_s}\big)^\ell,\;\;\;\;\; p_i\in\Spec(\mathcal{O}_{K,S})\] 
for some integers $e_i, q_i$ satisfying $v_{p_i}(d_n)=q_i\cdot\ell+e_i$ and $0\leq e_i<\ell$. In particular, by replacing $d_n$ with $\big(p_1^{e_1}\cdot p_2^{e_2}\cdots p_s^{e_s}\big)$ and $y_n$ with $\big(y_n\cdot p_1^{q_1}\cdot p_2^{q_2}\cdots p_s^{q_s}\big)$, we may assume that $0\leq v(d_n)\leq\ell-1$ for all $v\in V_K\mysetminus S$ as claimed.

Now suppose that $n\in Z(\phi,b,S,\ell)$. It is our goal to show that $n$ is bounded. To do this, first note that conditions (b) and (c) imply that $\phi$ has good reduction (see \cite[Thm. 2.15]{Silv-Dyn}) modulo the primes in $\Spec(\mathcal{O}_{K,S})$. In particular, if $p$ is such that $v_{p}(d_n)>0$ and $n\in Z(\phi,b,S,\ell)$, then $v_p(\phi^m(b))>0$ for some $1\leq m\leq n-1$. Moreover, 
\[\phi^{n-m}(0)\equiv\phi^{n-m}(\phi^{m}(b))\equiv\phi^n(b)\equiv0\Mod{p};\] see \cite[Thm. 2.18]{Silv-Dyn}. 
Therefore, we have the refinement  
\begin{equation}{\label{refinement}}  d_n=\prod p_i^{e_i},\;\;\text{where}\;\; p_i\big\vert\phi^{t_i}(b)\;\text{or}\; p_i\big\vert\phi^{t_i}(0)\;\; \text{for some}\; 1\leq t_i\leq\Big\lfloor \frac{n}{2}\Big\rfloor.  
\end{equation}
Moreover, as noted above, we may assume that $0\leq e_i\leq\ell-1$. Hence 
\begin{equation}{\label{htestimate}} \boxed{h_K(d_n)\leq (\ell-1)\cdot\bigg(\sum_{i=1}^{\lfloor\frac{n}{2}\rfloor}h_K(\phi^i(b))+ \sum_{j=1}^{\lfloor\frac{n}{2}\rfloor}h_K(\phi^j(0))\bigg)} 
\end{equation} 
Now, choose (and fix) an integer $m\geq1$ such that the curve 
\[C_{m,\ell}(\phi): Y^\ell=\phi^m(X)\] is nonsingular and of genus at least two - possible, since $\phi$ is dynamically $\ell$-power non-isotrivial. If $n\leq m$ for all $n\in\mathcal{Z}(\phi,b,\ell,S)$, then we are done. Otherwise, we may assume that $n>m$ so that (\ref{decomp}) implies that 
\[P_n:=\big(\phi^{n-m}(b)\;,\;y_n\cdot\sqrt[\ell]{u_n\cdot d_n}\, \big)\in C_m(\phi)\big(K\big(\sqrt[\ell]{u_n\cdot d_n}\big)\big)\;\;. \]
It follows from any of the bounds (suitable to positive characteristic) discussed in the introductions of \cite{Kim} or \cite{htineq} that there exist constants $A_1, A_2>0$ such that 
\begin{equation}{\label{Szpiro}}
h_{\kappa(\phi,m)}(P_n)\leq A_1\cdot d(P_n)+A_2,
\end{equation}  
where $\kappa(\phi,m)$ is the canonical divisor of $C_{m,\ell}(\phi)$, \[d(P_n):=\frac{2\cdot \genus\big(K_n\big)-2}{\big[K_n:K\big]},\;\;\;\;\text{and}\;\;\;K_n:=K\big(\sqrt[\ell]{u_n\cdot d_n}\,\big).\]
Crucially, the bounds $A_i=A_i(\phi,\ell,m)$ are independent of both $b$ and $n$. We note that the bounds on (\ref{Szpiro}) have been improved by Kim \cite[Theorem 1]{Kim}, although we do not need them here. 

On the other hand, it follows from \cite[Prop. III.7.3 and Remark III.7.5]{ffields} that $d(P_n)\leq B_1\cdot h_K(u_n\cdot d_n)+B_2$ for some positive constants $B_i=B_i(K)$. Likewise, $h_K(u_n\cdot d_n)\leq h_K(d_n)+B_3(K,S,\ell)$, since the height of $u_n$ is absolutely bounded. In particular, after combining these bounds with those on (\ref{Szpiro}), we see that  
\begin{equation}{\label{combin1}}
h_{\kappa(\phi,m)}(P_n)\leq D_1\cdot h_K(d_n)+D_2, \;\;\;\;\;\;\; D_i=D_i(K,\phi,S,\ell,m)>0.
\end{equation} 
However, if $\mathcal{D}_1$ is an ample divisor on $C_{m,\ell}(\phi)$ and $\mathcal{D}_2$ is an arbitrary divisor, then 
\begin{equation}{\label{Divisor}} \lim_{h_{\mathcal{D}_1}(P)\rightarrow\infty}\frac{h_{\mathcal{D}_2}(P)}{h_{\mathcal{D}_1}(P)}=\frac{\deg\mathcal{D}_2}{\deg{\mathcal{D}_1}},\;\;\;\;\;\;P\in C_{m}(\phi);
\end{equation}  
see \cite[Thm III.10.2]{SilvA}. In particular, if $\pi:C_{m,\ell}(\phi)\rightarrow\mathbb{P}^1$ is the covering $\pi(X,Y)=X$, then a degree one divisor on $\mathbb{P}^1$ (giving the usual height $h_K$ on projective space) pulls back to a $\deg(\pi)$ divisor $\mathcal{D}_2$ on $C$ satisfying $h_{\mathcal{D}_2}(P)=h_K(\pi(P))$. 

We deduce from (\ref{Divisor}) that there exist constants $\beta$ and $E=E(\phi,m,\ell)$ such that 
\begin{equation}{\label{limit}} h_{\kappa(\phi,m)}(P)>\beta\;\;\;\;\;\text{implies}\;\;\;\;\; h_K(\pi(P))\leq E\cdot h_{\kappa(\phi,m)}(P)+1 
\end{equation} 
for all $P\in C_m(\phi)(\bar{K})$. However, note that 
\begin{equation}{\label{finiteness}} 
T:=\big\{P_n\;\big\vert\; h_{\kappa(\phi,m)}(P_n)\leq\beta\}\subseteq\big\{P\in C_m(\phi)(\bar{K})\;\big\vert\;\; h_{\kappa(\phi,m)}(P)\leq\beta\;\;\,\text{and}\,\;\; [K(P):K]\leq\ell\big\}, 
\end{equation}
and the latter set is finite, since $\kappa(\phi,m)$ is ample; see \cite[Thm. 10.3]{SilvA}. Hence, (\ref{htestimate}),(\ref{combin1}), (\ref{limit}), and (\ref{finiteness}) together imply that 
\begin{equation}{\label{combin2}} 
h_K(\phi^{n-m}(b))\leq F_1\cdot \bigg(\sum_{i=1}^{\lfloor\frac{n}{2}\rfloor}h_K(\phi^i(b))+ \sum_{j=1}^{\lfloor\frac{n}{2}\rfloor}h_K(\phi^j(0))\bigg)+F_2
\end{equation}
 for all but finitely many $n\in Z(\phi,b,S,\ell)$ and some positive constants $F_i=F_i(K,\phi,S,\ell,m)$.
 
 On the other hand, it is well known that the canonical height $\hat{h}_\phi$ satisfies: 
 \[\text{(a).}\;\;\;\hat{h}_\phi=h_K+O(1)\;\;\;\;\;\;\;\;\;\;\;\; \text{(b).}\;\;\;\hat{h}_\phi(\phi^s(\alpha))=d^s\cdot \hat{h}_\phi(\alpha) \] for all $\alpha\in K$ and all integers $s\geq0$; see \cite[Thm. 3.20]{Silv-Dyn}. In particular, (\ref{combin2}) implies that
 \begin{equation}{\label{combin3}} 
d^{n-m}\cdot\hat{h}_\phi(b)\leq G_1\cdot\bigg(\sum_{i=1}^{\lfloor\frac{n}{2}\rfloor}d^i\cdot \hat{h}_\phi(b)+ \sum_{j=1}^{\lfloor\frac{n}{2}\rfloor}d^j\cdot\hat{h}_\phi(0) \bigg)+G_2\cdot n+G_3 
\end{equation}
for almost all $n\in Z(\phi,b,S,\ell)$ (those $n$ such that $P_n\notin T$) and some constants $G_i=G_i(K,\phi,S,\ell,m)$. In particular, for almost all $n\in Z(\phi,b,S,\ell)$,
\[d^{n-m}\leq G\cdot \big(d^{\lfloor\frac{n}{2}\rfloor+1}+n+1\big),\;\;\;\;\text{where}\;\;G:=\max\bigg\{G_1,\frac{G_1\cdot \hat{h}_\phi(0)}{\hat{h}_\phi(b)},\frac{G_2}{\hat{h}_\phi(b)},\frac{G_2}{\hat{h}_\phi(b)}\bigg\}.\] However, since $m$ is fixed, such $n$ are bounded.            
\end{proof}
\end{section} 
\begin{section}{Dynamical Galois Groups}
Our main interest in proving the finiteness of $\mathcal{Z}(\phi,b,\ell)$ comes from the Galois theory of iterates. In particular, if $\phi(x)=(x-\gamma)^2+c$ is a quadratic polynomial and $\mathcal{Z}(\phi,\gamma,2)$ is finite, then the Galois groups of $\phi^n$ are large \cite[Theorem 3.3]{RafeThesis}, enabling us to compute the density of prime divisors $\mathcal{P}_\phi(b)$ (for all $b\in K$) via a suitable Chebotarev density theorem \cite{Eventually, Jones}. 

To define the relevant dynamical Galois groups, let $\phi$ be a polynomial and assume that $\phi^n$ is separable for all $n\geq1$; hence, the set $T_n(\phi)$ of roots of $\phi, \phi^2,\dots ,\phi^n$ together with $0$, carries a natural $\deg(\phi)$-ary rooted tree structure: $\alpha,\beta\in T_n(\phi)$ share an edge if and only if $\phi(\alpha) =\beta$ or $\phi(\beta)=\alpha$. Furthermore, let $K_n:=K(T_n(\phi))$ and $G_n(\phi):=\Gal(K_n/K)$. Finally, we set
\begin{equation}{\label{Arboreal}}
T_\infty(\phi):=\bigcup _{n \geq 0} T_n(\phi)\;\;\text{and}\;\; G_\infty(\phi)=\varprojlim G_n(\phi).
\end{equation}  
Since $\phi$ is a polynomial with coefficients in $K$, it follows that $G_n(\phi)$ acts via graph automorphisms on $T_n(\phi)$. Hence, we have injections $G_n(\phi) \hookrightarrow \Aut(T_n(\phi))$ and $G_\infty(\phi) \hookrightarrow \Aut(T_\infty(\phi))$ called the \emph{arboreal representations} associated to $\phi$.   
 
A major problem in dynamical Galois theory, especially over global fields, is to understand the size of $G_\infty(\phi)$ in $\Aut(T_\infty(\phi))$; see \cite{B-J,Me,R.Jones,Odoni,Stoll}. We now use Theorem \ref{PrimDivThm} to prove a finite index result for many quadratic polynomials \big(c.f. \cite[Theorem 1]{uniformity}\big), including the family defined in Corollary \ref{eg}, and provide an outline for further examples.
\begin{proof}[Proof (Corollary 1)] Let $\phi(x)=(x-\gamma)^2+c$ and let $m\geq2$. Then we have a map 
\begin{equation}{\label{map}} \Phi_m:C_{2,m}(\phi)\rightarrow E_\phi,\;\;\;\; \Phi(x,y)=\big(\phi^{m-1}(x),\;y\cdot(\phi^{m-2}(x)-\gamma)\big). 
\end{equation} 
It follows from Proposition \ref{prop} below that $C_{2,m}$ is non-isotrivial. On the other hand, since $\widebar{\mathcal{O}}_\phi(\gamma)$ contains no squares in $K$, \cite[Proposition 4.2]{Jones} implies that $\phi^n$ is an irreducible polynomial over $K$ for all $n$; hence,  $C_{2,m}(\phi)$ and $E_\phi$ are non-singular; see \cite[Lemma 2.6]{Jones}. Therefore, we may choose $m$ so that $C_{2,m}(\phi)$ is a non-isotrivial curve of genus at least $2$. In particular, $\phi$ is dynamically 2-power non-isotrivial and Theorem \ref{PrimDivThm} implies that $\mathcal{Z}(\phi,b,2)$ is finite for all wandering $b\in K$.

For the second claim, we apply this fact to $b=\gamma$ and use \cite[Theorem 3.3]{RafeThesis} to deduce that $G_\infty(\phi)\leq\Aut(T(\phi))$ is a finite index subgroup. Finally, \cite[Theorem 4.2]{R.Jones} an \cite[Theorem 1.3]{RafeThesis} imply that the density of $\mathcal{P}_\phi(b)$ is zero for all $b\in K$.         
\end{proof} 
We now apply Corollary \ref{Galois} to the family $\phi_{(f,g)}$ defined in Corollary \ref{eg}.   
\begin{proof}[Proof (Corollary 2)] Let $K=\mathbb{F}_q(t)$ and let $\phi(x)=\phi_{(f,g)}=(x-f(t)\cdot g(t)^d)^2+g(t)$. It suffices to check conditions (a) and (b) of Corollary \ref{Galois} hold to prove Corollary \ref{eg}. We first show that the adjusted critical orbit of $\phi$, the set $\{-\phi(f\cdot g^d),\phi^2(f\cdot g^d),\phi^3(f\cdot g^d),\dots\}$, contains no perfect squares in $K$; in particular, $\phi^n$ is an irreducible polynomial over $K$ for all $n\geq1$; see \cite[Proposition 4.2]{Jones}. 

Note that $-\phi(f\cdot g^d)=-g$ is not a square in $K$ by assumption. On the other hand, we let $h:=g-f\cdot g^d$ and suppose that 
\begin{equation}{\label{iterate}} j^2=\phi^n(f\cdot g^d)=((((h^2+h)^2+h)^2+h)^2+\dots+h)^2+g
\end{equation} 
for some polynomial $j\in \mathbb{F}_q[t]$ and some $n\geq2$. Hence, $j^2=g^2\cdot k^2+g=g\cdot(g\cdot k^2+1)$ for some $k\in \mathbb{F}_q[t]$, since $g\vert h$. However, because $\mathbb{F}_q[t]$ is a UFD and $g$ and $g\cdot k^2+1$ are coprime, it follows that $g=l^2$ and $g\cdot h^2+1=m^2$ for some $l,m\in \mathbb{F}_q[t]$. In particular, $1=(m+l\cdot h)(m-l\cdot h)$ and both factors are constant. Hence, $2m=(m+l\cdot h)+(m-l\cdot h)$ is also constant. Finally, since $m^2-1=g\cdot h^2$ and $h=g-f\cdot g^d$, it follows that $g, h$, and $f$ are all constant. In particular, the $j$-invariant of $E_\phi$ is constant,  a contradiction. 

On the other hand, the right hand side of (\ref{iterate}) implies that \[\deg(\phi^n(f\cdot g^d))\leq\max\big\{2^{n-1}\cdot\deg(h),\deg(g)\big\}=\max\big\{2^{n-1}\cdot[\deg(g)+\deg(1-f\cdot g^{d-1})],\deg(g)\big\},\] with equality if the terms are unequal. In particular, if $\deg(h)\neq 0$, then we may choose $n$ large enough so that $\deg(\phi^n(f\cdot g^d))=2^{n-1}\cdot\deg(h)\geq 2^{n-1}$, and $\phi$ is post-critically infinite. Otherwise, we may assume that $h$ is constant. Consequently, since $h=g\cdot(1-f\cdot g^{d-1})$, we deduce that either $g$ is constant or $1-f\cdot g^{d-1}=0$. However, $g$ and $h$ constant implies that $f$ is constant, a contradiction. We deduce that $1-f\cdot g^{d-1}=0$. In particular, $g=f\cdot g^d$ and $\phi(x)=(x-g)^2+g$. In this case the elliptic curve 
\[E_\phi:\;Y^2=(X-g)\cdot\phi(X)=(X-g)\cdot\big((X-g)^2+g)\] 
has $j$-invariant $1728$, contradicting our assumption that $E_\phi$ have non-constant $j$-invariant.                          
\end{proof} 
We expect that most $\phi\in K[x]$ are dynamically $\ell$-power non-isotrivial for some $\ell$, and we state a conjecture along these lines:    
\begin{conjecture} Suppose that $\phi\in K[x]\mysetminus \widebar{\mathbb{F}}_q[x]$ satisfies the following conditions: 
\begin{enumerate} 
\item[\textup{(1)}] $\deg(\phi)\geq2$,
\item[\textup{(2)}] $\phi$ is not conjugate to $x^d$ (as a function on $\mathbb{P}^1$)
\item[\textup{(3)}] $\gcd(\deg(\phi),q)=1$,
\item[\textup{(4)}] $0\notin \mathcal{O}_\phi(\gamma)$ for all critical points $\gamma\in\bar{K}$ of $\phi$. 
\end{enumerate}
Then there exists $\ell\geq2$ such that $\phi$ is dynamically $\ell$-power non-isotrivial  and $\gcd(\ell,q)=1$.   
\end{conjecture}
\begin{remark} Conditions (3) and (4) imply that the discriminant of $\phi^n$ is non-zero for all $n$; see \cite[Lemma 2.6]{Jones}. In particular, the curves $C_{m,\ell}(\phi)$ are non-singular and eventually of large genus, since $\deg(\phi^m)$ grows exponentially by condition (1).       
\end{remark} 
We finish by proving an auxiliary result used in the proof of Corollary \ref{eg}. The argument below was suggested by Bjorn Poonen. 
\begin{proposition}{\label{prop}} Let $\Phi: C_1\rightarrow C_2$ be a non-constant morphism. If $C_1$ is isotrivial, then $C_2$ is isotrivial. 
\end{proposition} 
\begin{proof}Suppose not. That is, suppose that $C_1$ is isotrivial but $C_2$ is not.
By replacing $K$ by a finite extension, we may assume that $C_1$ is constant. We may also assume that $\Phi$ is separable (if not, then it factors as a power of Frobenius composed with a separable
morphism, say $C_1 \rightarrow C_1' \rightarrow C_2$,
and then (maybe after a finite extension) 
$C_1'$ is isomorphic to the curve obtained by taking the $p^n$-roots
of all the coefficients of $C_1$, so $C_1'$ is constant too,
and is separable over $C_2$; rename $C_1'$ as $C_1$).

We fix some notation. Let $g_i$ be the genus of $C_i$, and let $J_i$ be the Jacobian of $C_i$. Since $C_2$ is non-isotrivial, $g_2>0$. 

Let $\Omega$ be an uncountable algebraically closed field containing $F_q$.
Specializing (by choosing an $F_q$-homomorphism $\sigma: K \rightarrow \Omega$)
gives separable maps over $\Omega$ from the same $C_1$ to 
uncountably many non-isomorphic curves $C_2^\sigma$.  
Each isogeny class of an elliptic curve over $\Omega$
consists of at most countably many elliptic curves,
so if $g_2=1$, this would imply that in the decomposition of $J_i$ 
up to isogeny, uncountably many isogeny factors occur,
which is impossible.
If $g_2>1$, then the infinitely many maps from $C_1$ to various curves $C_2^\sigma$
contradict \cite[Theorem 2]{Samuel}. 
\end{proof} 
\begin{remark} With some additional hypothesis, Proposition \ref{prop} holds for projective varieties of arbitrary dimension \cite{isotrivial}. Such a result is necessary if one hopes to generalize Theorem \ref{PrimDivThm} as in \cite{Silv-Vojta}    
\end{remark} 
\textbf{Acknowledgements:} It is a pleasure to thank Joseph Silverman, Bjorn Poonen, Felipe Voloch, and Rafe Jones for the discussions related to the work in this paper.
\end{section} 

\end{document}